\documentclass[11 pt]{amsart}

\usepackage[lite]{amsrefs}
\usepackage{amsmath}
\usepackage{latexsym,amsfonts,amssymb,mathrsfs}
\usepackage[all,cmtip]{xy}

\usepackage[
textwidth=3cm,
textsize=small,
colorinlistoftodos]
{todonotes}
\usepackage{tikz-cd}
\tikzset{shorten <>/.style={shorten >=#1,shorten <=#1}}
\tikzset{%
    symbol/.style={%
        ,draw=none
        ,every to/.append style={%
            ,edge node={%
                node [%
                    ,sloped
                    ,allow upside down
                    ,auto=false
                    ]{$#1$}
                }
            }
        }
    }

\newcommand{\cA}{\mathcal{A}}
\newcommand{\cB}{\mathcal{B}}
\newcommand{\cC}{\mathcal{C}}
\newcommand{\cD}{\mathcal{D}}

\newcommand{\cM}{\mathcal{M}}
\newcommand{\cN}{\mathcal{N}}

\newcommand{\cP}{\mathcal{P}}

\newcommand{\tensor}{\otimes}

\newcommand{\mtens}{\mathcal{M}^{\tensor}}

\newcommand{\mult}{Mult_*}
\newcommand{\nmult}{NMult_*}

\newcommand{\G}{\mathcal{G}}

\newcommand{\M}{\mathcal{M}}
\newcommand{\N}{\mathcal{N}}

\numberwithin{equation}{section}

\theoremstyle{plain}
\newtheorem{thm}[equation]{Theorem}
\newtheorem{cor}[equation]{Corollary}
\newtheorem{lem}[equation]{Lemma}
\newtheorem{prop}[equation]{Proposition}

\theoremstyle{definition}
\newtheorem{defn}[equation]{Definition}

\theoremstyle{remark}
\newtheorem{rem}[equation]{Remark}
\newtheorem{ex}[equation]{Example}

\begin{document}
\title{Modeling Connective Spectra via Multicategories}
\author[D.~Fuentes-Keuthan]{Daniel Fuentes-Keuthan}
\address{Department of Mathematics,
Johns Hopkins University, Baltimore
}
\email{danielfk@jhu.edu}
\address{}
\date{}
\begin{abstract}
    We put a model structure on a full subcategory of based multicategories in which the weak equivalences are created by the $K$-theory functor of Elmendorf-Mandell, providing a model categorical lift of Thomason's theorem on the modeling of connective spectra by symmetric monoidal categories. We note that this lifts to a semi-model structure on based multicategories itself. As a corollary we show that to model connective spectra up to stable equivalence it suffices to restrict to symmetric monoidal groupoids.
\end{abstract}
\maketitle
\tableofcontents



\section*{Introduction}\label{secone}

In \cite{S}, Segal introduced $\Gamma$-spaces as a combinatorial model for connective spectra, proving an equivalence between the homotopy category of connective spectra and a suitable homotopy category of $\Gamma$-spaces in which the weak equivalences are the so called stable ones. A $\Gamma$-space is a contravariant functor from Segal's category $\Gamma$ to pointed simplicial sets. The category $\Gamma$ is a skeletal category of finite pointed sets and so the category  of  $\Gamma$-spaces is equivalent to the category $sSet_*^{Fin_*}$ of functors from finite pointed sets to pointed simplicial sets, but we will use the standard  name $\Gamma\textit{-sSet}$ to refer to this category. In \cite{BF} Bousfield and Frielander showed that the generalized Reedy model structure on $\Gamma\textit{-sSet}$ could be localized to expand the class of weak equivalences to the stable equivalences of Segal. This put Segal's theorem in a model categorical context and provided a fully fledged homotopy theory of connective spectra.

In \cite{S} (see also \cite{EM2}) Segal also describes a process for obtaining a $\Gamma$-spaces, and hence also a connective spectrum, from a small symmetric monoidal category.  When applied to the category of finitely generated projective modules over a commutative ring this spectrum agrees with the usual algebraic $K$-theory.  For this reason Segal's construction is referred to as the $K$-theory of a symmetric monoidal category.  This naturally raises the question: which connective spectra arise (up to stable equivalence) from this construction?

This question was answered by Thomason \cite{Thom}, and refined by Mandell \cite{M}. In fact all connective spectra arise in this way. If we consider the category $SymMonCat$ of symmetric monoidal categories and lax monoidal functors as a category with weak equivalences created by Segal's $K$-theory then we obtain

\begin{thm}[Thomason]\label{thomthm}
   Segal's construction induces an equivalence of homotopy categories \[ SymMonCat\rightarrow ConnSpectra \]
   In particular every connective spectrum is related by stable equivalence to the $K$-theory spectrum of a symmetric monoidal category.
\end{thm}

Earlier Thomason \cite{Thom2} proved a similar theorem in unstable homotopy theory, an equivalence between the homotopy category of spaces and a suitable homotopy category of small categories, by introducing the Thomason model structure on Cat via a right transfer of the Kan model structure and proving that the resulting Quillen adjunction was a Quillen equivalence. Hence Theorem \ref{thomthm} is almost the complete analog in stable homotopy theory of this result, except that it is a statement on the level of homotopy categories, rather than at the level of a model structure on $SymMonCat$. 

Our main result is Theorem \ref{modelstructure}, which aims to fill this gap by establishing a model structure on a category similar to $SymMonCat$ which is Quillen Equivalent to a suitable model of connective spectra. We detail these structures now.

The category $SymMonCat$ is not cocomplete, and so cannot carry a model structure. In \cite{EM} Elmendorf and Mandell extend the domain of their refinement of Segal's $K$-theory construction developed in \cite{EM2} to the category of based multicategories, $\mult$. We detail this extension in the next section.  Among other benefits recounted in \cite{EM}, this category serves as a sort of colimit completion of the category $SymMonCat$, fixing the first obstruction to capturing Theorem \ref{thomthm} at the level of model categories.

In \cite{Sc} Schwede considers a Quillen equivalent model structure on $\Gamma\textit{-sSet}$ in which one starts with the projective model structure and localizes again with respect to the same stable equivalences.  This model structure is referred to as the stable Q-model structure, and for our purposes will be taken as our model of the homotopy theory of connective spectra. In the Q-model structure the property of being a fibration is easier to check than in the Bousfield-Friedlander model structure, which will be beneficial to us. We summarize the properties of this model structure;

\begin{defn} We make use of the following vocabulary for a given natural transformation $\alpha:X_{\bullet}\rightarrow Y_{\bullet}$ of $\Gamma\text{-spaces}$.
\begin{itemize}
    \item A strict Q-weak equivalence is a levelwise weak equivalence of simplicial sets.
    \item A strict Q-fibration is a levelwise fibration.
    \item A strict Q-cofibration is a morphism with the left lifting property against any map which is a trivial strict Q-fibration.
\end{itemize}
In the stable case we have

\begin{itemize}
    \item A stable Q-weak equivalence is a morphism which induced isomorphisms of homotopy groups for the connective spectra associated to $X_{\bullet}, Y_{\bullet}$ in \cite{S}.
    \item A stable Q-cofibration is a strict Q-cofibration.
    \item A stable Q-fibration is morphism which has the right lifting property against any map which is a trivial stable Q-cofibration.
\end{itemize}
\end{defn}

The notions of strict Q-cofibration, fibration, and weak equivalence make $\Gamma\textit{-sSet}$ into a simplicial model category. Likewise we have

\begin{thm}[Schwede]
The stable Q-cofibrations, Q-fibrations and Q-equivalences make the category $\Gamma\textit{-sSet}$ into a cofibrantly generated simplicial model category. A $\Gamma$-space X is stably Q-fibrant if and only if
\begin{itemize}
    \item For each $n$, $X_n$ is a Kan complex.
    \item $\pi_0(X_1)$ is an abelian group.
    \item $X$ satisfies the Segal condition, for each $n$ there are equivalences \[X_n\xrightarrow{\simeq} X_1\times X_1\times\cdots X_1\]
    
\end{itemize}
Furthermore, a strict Q-fibration between stably Q-fibrant $\Gamma$-sSets is a stable Q-fibration.
\end{thm}

\begin{rem}
When a $\Gamma$-space satisfies the third conditions above it is often called special, and if it satisfies the second and third condition it is called very special. Thus the stably fibrant objects are the levelwise fibrant very special $\Gamma$-sSets.
\end{rem}

The breakdown of this paper is as follows:

\begin{enumerate}

\item In Section 1 we recall definitions and facts about (based) mutlicategories and their accompanying $K$-theory spectra as studies in \cite{EM}.

\item In Section 2 we make a digression to study the homtopy theory of the category $\G_*\textit{-sSet}$ of symmetric functors introduced in \cite{EM2} to define the $K$-theory of multifunctors. We establish the existence of a model structure on $\G_*\textit{-sSet}$ which is Quillen equivalent to the stable Q-model structure. This model structure may be of future independent interest.

\item In Section 3 we study a full subcategory of $\mult$, which we call $Mod_E$. Objects in this subcategory have a particularly nice form, identified in Proposition \ref{modstruc1}.

\item In Section 4 we establish our first main theorem, Theorem \ref{modelstructure}, which states the existence of a model structure on $Mod_E$,  The $K$-theory spectrum of every based multicategory is \textit{isomorphic} to the $K$-theory of an object of $Mod_E$ so in a very strict sense no information is lost in this restriction. Indeed our second main theorem, Theorem \ref{main2}, establishes this formally by showing that the $K$-theory functor of \cite{EM} is a right Quillen equivalence between our model structure on $Mod_E$ and the stable Q-model structure on $\Gamma$-sSet, providing the desired model categorical lift of Theorem \ref{thomthm}.

We note also that the model structure of Theorem \ref{modelstructure} lifts naturally to a \textit{semi}-model structure (see Def \ref{semimodeldef}) on $\mult$ itself. While weaker than a model structure, a semi-model structure provides all of the usual model categorical tools when one works amongst the cofibrant objects. Though our techniques only allow us to establish this weaker notion, it is possible that this semi-model structure is in fact a full model structure.

\item In Section 5 we examine the fibrant objects of our model structure on $Mod_E$ and obtain a refinement of Thomason's theorem by showing that every connective spectrum is stably equivalent to the $K$-theory of a symmetric monoidal \textit{groupoid}.  This is our Theorem \ref{symgroup}.

\end{enumerate}

\textbf{Acknowledgements.}
The authors is grateful to his advisor Emily Riehl for many helpful discussions on this topic and for reading several drafts of this article.

\section{The $K$-Theory of Based Multicategories}
In this section we briefly recall the theory of multicategories and the $K$-theory construction of Elmendorf-Mandell given in \cite{EM}, where proofs and detailed exposition of what follows may be found.

\begin{defn}
A multicategory $\M$ consists of
\begin{itemize}
    \item A collection of objects
    \item For each collection of objects $a_1,\dots,a_n, b$ with $n\ge0$ a set \[M(a_1,\dots,a_n:b)\] of ``n-arrows", which we also call the multi-homsets or multi-morphism sets.
    \item For each collection of objects $a_1,\dots,a_n, b$ with $n\ge0$ and  permutation $\sigma\in \Sigma_n$ an isomorphism \[\sigma^*: \M(a_1,\dots,a_n:b)\xrightarrow{\cong} \M(a_{\sigma(1)},\dots,a_{\sigma(n)}:b)\]
    \item Composition maps
    \[\M(b_1,\dots, b_n: c)\times\M(a^1_1,\dots,a^1_{m_1}:b_1)\times\cdots\times \M(a^n_1,\dots, a^n_{m_n}: b_n)\rightarrow \M(a^1_1,\dots,a^n_{m_n}:c)\]
    \item For each $a$ a unit element $1_a\in \M(a:a)$ for the composition.
    
    In addition the composition is required to be appropriately associative and respects permutations and units.
\end{itemize}
\end{defn}
\begin{rem}
What we have defined here ought to be called a \textit{symmetric} multicategory, but we will use the shorter term multicategory instead.
\end{rem}

Every multicategory has an underlying category by forgetting the $n$-arrows for $n$ not equal to 1, and every category can be made into a multicategory by letting the multimorphism sets be empty when the source has less or more than one object.
The natural morphisms between multicategories are \textit{multifunctors} which are morphisms that preserve all of the above structure.

\begin{ex} Every symmetric monoidal category $(\M,\tensor)$ can be considered as a multicategory with $\M(a_1,\dots,a_n: b) = \M(a_1\tensor\cdots\tensor a_n, b)$.  In general a multicategory can be thought of us a ``virtual" symmetric monoidal category, where we do not actually require objects $x\tensor y$ to exist in our category, but to which we can refer to externally via the multi-hom functors $\M(x, y: -)$.
\end{ex}

Many of the familiar constructions involving symmetric monoidal categories carry over to multicategories.
\begin{defn}
A commutative monoid in a multicategory $\M$ is an object $a\in \M$ together with a distinguished element $\mu_n\in M(a,\dots, a: a)$ for each $n$, such that the composition maps preserve the collection $\{\mu_m\}$ in the sense that for each $n, m_1,\dots,m_n$ we have \[\mu_n\circ(\mu_{m_1},\dots, \mu_{m_n}) = \mu_{m_1+\cdots+m_n}\]
\end{defn}
The idea is that in a commutative monoid $a$ there is a non ambiguous way to ``multiply $n$ elements of a" due to associativity and commutativity.  Because all monoids we consider will be commutative we usually omit this term and simply refer to them as monoids.

Because multicategories can be difficult to get a handle on, it will be useful to have external characterizations of objects like monoids.  Let $*$ be the terminal multicategory, consisting of one object and a single $n$-arrow for each $n$.  With a little thought one sees

\begin{lem}
    There is a bijection between commutative monoid objects in $\M$ and multifunctors $*\rightarrow \M$.
\end{lem}

We will call a multicategory \textit{based} if it comes with a distinguished multifunctor $*\rightarrow \M$ from the terminal multicategory. In particular the above lemma tells us that a based multicategory is a multicategory with a preferred monoid object.  Based multicategories form a category $\mult$ with morphisms those multifunctors which preserve the structure of the basepoint monoid.

\begin{rem}
Note that we are not requiring our multicategories to be based at an initial, terminal, or zero object, but simply at some monoid.
\end{rem}

\begin{defn}
Given a multicategory $\M$ and a monoid object $a\in \M$, an $a$-module object is an object $m\in \M$ together with a distinguished $(n+1)$-arrow $\mu_n\in M(a,\dots,a,m:m)$  for each $n\ge 1$ which are suitably compatible under the composition.
\end{defn}

Rather than spell out the specific composition compatibility conditions of modules, we can describe module objects externally via multifunctors out of a multicategory E that parameterizing monoid-module pairs. The multicategory E consists of a monoid together with an additional object that is a module over that monoid.  Rigorously E has two objects 0 and 1, and the multimorphism sets \[E(a_1,\dots,a_n:b) = \{*\text{ if }\Sigma a_i = b, \emptyset\text{ otherwise} \} \]

When $\M$ is based, so that it comes with a preferred monoid, we will call a module over this monoid a module in $\M$.

\begin{lem}
    Considering E as a multicategory based at 0, there is a bijection between modules in $\M$ and based multifunctors $E\rightarrow \M$.
\end{lem}

We will be interested in a symmetric monoidal structure on $\mult$ which characterizes bilinear maps.
\begin{defn}
Given multicategories $\M$, $\N$ and $\cP$, the data of a bilinear map $f:(\M,\N)\rightarrow \cP$ is given by
\begin{enumerate}
    \item A function $f:ob(\M)\times ob(\N)\rightarrow ob(\cP)$
    \item For each m-arrow $(a_1,\dots,a_m)\rightarrow a$ of $\M$ and object $b$ of $\N$, an m-arrow \[(f(a_1,b),\dots, f(a_n, b))\rightarrow f(a,b) \]
    \item For each $n$-arrow $(b_1,\dots,b_m)\rightarrow b$ of $\N$ and object $a$ of $\M$, an n-arrow \[(f(a, b_1),\dots, f(a, b_n))\rightarrow f(a,b) \]
\end{enumerate}
Such that
\begin{enumerate}
    \item[(A1)] For each fixed $a$ in $\M$ the map $f(a,-):\N\rightarrow \cP$ is a multifunctor.
    \item[(A2)] For each fixed $b$ in $\N$ the map $f(-, b):\M\rightarrow \cP$ is a multifunctor.
    \item[(A3)] For a given $m$ and $n$ arrows $(a_1,\dots,a_m)\rightarrow a$ and $(b_1,\dots,b_m)\rightarrow b$ in $\M$ and $\N$ respectively, the two ways to produce an $nm$-arrow \[(f(a_1, b_1),\dots, f(a_n,b_m))\rightarrow f(a,b)\] from (1) and (2) above are equal. 
\end{enumerate}
\end{defn}
We call a bilinear map based if the multifunctors $f(a,-), f(-,b)$ factor through the basepoint of $\cP$ when $a$ or $b$ is the basepoint of $\M$ or $\N$. The following theorem summarizes the results we will need about based multicategories from \cite{EM}.

\begin{thm}[Elmendorf-Mandell] The category $\mult$ is a $\mult$-enriched, tensored and cotensored, complete, and cocomplete symmetric monoidal category.  The symmetric monoidal product, denoted $\wedge$ represents based bilinear functors \[\mult(\M \wedge \N, \cP)\cong Bilin_*(\M, \N: \cP) \cong \mult(\M, \cP^{\N})\]
By forgetting along the lax monoidal functors \[\mult\rightarrow Mult\rightarrow Cat \] we obtain a categorical enrichment on $\mult$.
\end{thm}

\begin{rem}
Explicitly, the category $\mult(\M, \N)$ is the underlying category of the internal hom multicategory $\N^{\M}$, and has as its objects the based multifunctors, and as its arrows the based \textit{multinatural transformations}, which are a direct generalization of natural transformation to this setting.

\end{rem}

\begin{rem} Notably the symmetric monoidal product on $\mult$ is \textit{not} derived from the cartesian product of multicategories.
\end{rem}

The $K$-theory construction of \cite{EM} for based multicategories comes in two parts: first the authors construct a category $\G_*$ and a functor $J:\mult\rightarrow \G_*\textit{-Cat}$, and then a functor $\G_*\textit{-Cat}\rightarrow Spectra$.  We will only be interested in the functor $J$, but we choose to forgo a discussion of the category $\G_*$ until the next section, because we will bypass it via a composition $\mult\xrightarrow{J}\G_*\textit{-Cat}\xrightarrow{i^*}\Gamma\textit{-Cat}$, where $i:\Gamma\hookrightarrow \G_*$ is a certain fully faithful inclusion. By abuse of notation we also call this composition $J$.

By construction based multifunctors $E\rightarrow \M$ pick out module objects over the basepoint monoid of $\M$. Taking cartesian powers of E now gives a functor $E^{\bullet}:\Gamma^{op}\rightarrow \mult$, as proven in \cite[Thm 5.14]{EM}

\begin{defn}
The $K$-theory functor $J:\mult\rightarrow \Gamma\textit{-Cat}$ is given by $\mult(E^\bullet, -)$.
\end{defn}

We will often make use of the pushforward of this functor along the nerve functor, which we denote $NJ:\mult\rightarrow \Gamma\textit{-sSet}$.

\section{The Homotopy Theory of Symmetric Functors}\label{symfunc}
In this section we recall the category $\G_*$ of \cite{EM2} and study its homotopy theory. More precisely, we will see that there is a model structure on $\G_*$-sSet which is Quillen Equivalent to the stable Q-model structure. The results of this section will not be needed for our main theorems, but we include them here in the interest of future work.

The category $\G_*$ is constructed as the Grothendieck construction of a certain functor, but since a complete account is already contained in the original reference \cite{EM} we will give only a brief overview.
\begin{defn}
We denote by $n$ the finite, nonempty set ${0,1,\dots,n}$ containing $n+1$ elements.
\end{defn}

\begin{defn}
The category $\G_*$ has as objects finite lists of finite sets $(n_1,\dots,n_m)$, except that we identify $(0)$ with any list $(n_1,\dots,n_m)$ in which at least one of the $n_i$ is $0$.

A morphism in $\G_*$ from $(n_1,\dots,n_m)\rightarrow (k_1,\dots,k_t)$ is given by an injection $f:m\rightarrow t$ and a collection of maps $\alpha_i:n_{q^{-1}(i)}\rightarrow k_i$ for $1\le i\le t$, where by convention if $q^{-1}(i) = \emptyset$ then $n_{q^{-1}(i)} = 1$, the finite set with 2 elements. We denote such a morphism by $(f, \{\alpha_i\})$. In addition we quotient out by all morphisms such that any of the map $n_{q^{-1}(i)}\rightarrow k_i$ in the collection factors through 0.
\end{defn}

The above definition makes $\G_*$ into a based category at the zero object 0, and makes the homsets themselves into based sets. From the definition of morphisms we have

\begin{lem}
    The functor $i:\Gamma\rightarrow \G_*$ which sends an object $n$ to the tuple $(n)$ and a morphism $n\xrightarrow{f} m$ to $(1_0, \{f\})$ is fully faithful.
\end{lem}

Recall Kan's theorem on the transfer of model structures along right adjoints.
\begin{thm}[Kan]
\label{theoremtransferredmodelstructure}
Let $\cM$ be a cofibrantly generated model category with a  set of generating cofibrations $I$ and a set of generating acyclic
cofibrations $J$, let $\cN$ be a complete and cocomplete category, and let
$$F\colon\cM\rightleftarrows\cN\colon G$$
be an adjunction.
Suppose that the following conditions hold:
\begin{enumerate}
\item The left adjoint $F$ preserves small objects; this is the case in particular when the right adjoint preserves filtered colimits.

\item The right adjoint $G$ takes maps that can be constructed
as a transfinite composition of pushouts of elements of $F(J)$ to weak equivalences.

\end{enumerate}
Then, there is a cofibrantly generated model category structure on $\cN$ in which
\begin{itemize}
\item the set $F(I)$ is
a set of generating cofibrations,
\item the set $F(J)$ is a set of generating acyclic cofibrations,
\item the weak equivalences are the maps that $G$ takes to weak equivalences in $\cM$ and
\item the fibrations are the maps that $G$ takes to fibrations in $\cM$.

\end{itemize}
Furthermore, with respect to this model category structure, the adjunction $(F, G)$ is a Quillen pair.
\end{thm}

To establish the model structure on $\G_*\textit{-sSet}$ we make use of the following general lemma on right transfer of model structure.

\begin{lem} Suppose $\cN$ is a bicomplete category, $\cM$ admits a cofibrantly generated model structure, and we are given a functor $i\colon \cN\to \cM$ which admits left and right adjoints $R\vdash i\vdash L$. Then the right transferred model structure on $\cN$ along $i$ exists.  Suppose further that $L$ is fully faithful, then $i$ is a right Quillen equivalence.
\end{lem}
    \begin{proof}
    Because $i$ preserves all colimits it preserves filtered colimits. Likewise, because transfinite composites of pushouts of cofibrations of trivial cofibrations are a type of colimit, the acyclicity condition boils down to $i$ preserving weak equivalences, which it does by definition in the transferred model structure. For the second statement, recall that for a right transferred model structure to be Quillen equivalent to the original structure, we need only check that the unit is a weak equivalence, but since $L$ is fully faithful, the unit is in fact an isomorphism.
    \end{proof}

\begin{rem}
It is a standard categorical result that if one of the outer functors of an adjoint triple is fully faithful that the other must be as well, and so in the above lemma we could have also assumed $R$ to be fully faithful. See for example the nLab article on adjoint triples.
\end{rem}

All we will need to apply this Lemma is the existence of the fully faithful inclusion $i:\Gamma\hookrightarrow \G_*$.

\begin{prop}\label{gstar} There is a cofibrantly generated model structure on the category $\G_*\textit{-sSet}$ which is Quillen equivalent via the functor $i^*$ to the Q model structure on $\Gamma\textit{-sSet}$
\end{prop}
    \begin{proof}
    The fully faithful inclusion $i\colon \Gamma \hookrightarrow \G_*$ admits a left adjoint $c$ which sends everything not in the image of $i$ to the basepoint. Since $i$ is fully faithful, the right adjoint to $i^*$ \[c^*\colon \Gamma\textit{-sSet}\hookrightarrow \G_*\textit{-sSet}\] is as well, and hence so is the left kan extension functor \[L\colon \Gamma\textit{-sSet}\hookrightarrow \G_*\textit{-sSet}\] The lemma above now establishes the desired transferred equivalent model structure.
    \end{proof}
    
While we will not be interested in this model structure for its own sake, considering that \cite{EM2} introduced $\G_*\textit{-sSet}$ to study multiplicative structures in $K$-theory, we believe it could be useful for that purpose in the future.

\section{The Category of E-modules}
We begin with the following proposition from \cite{EM}.

\begin{prop}[Elmendorf-Mandell]
E is a commutative monoid object in the symmetric monoidal category $\mult$.
\end{prop}

In this section we study the category of modules over E, $Mod_E$, which is a full subcategory of $\mult$. In particular we will show that $Mod_E$ is a reflective and coreflective subcategory. It is on this well-behaved subcategory that we will place a model structure.

\begin{thm}[Elmendorf-Mandell]\label{Shulman} The category of E-modules is a full symmetric monoidal Cat enriched subcategory of $\mult$, with monoidal product $\wedge_E \cong \wedge$. For any E-module $\cM$, the module structure map $\M\wedge E\to \cM$ is necessarily an isomorphism.

\end{thm}

The adjoint to the module structure map is also an isomorphism. The proof is similar to that of Theorem \ref{Shulman}.

\begin{prop}\label{modstruc1}
    For any E-module $\cM$, there is an isomorphism $\cM\rightarrow \cM^E$. In particular for an E-module $\cM$ we have isomorphisms \[\M\wedge E\cong \cM\cong \cM^E\]
\end{prop}
\begin{proof}
Let $u$ denote the based multicategory $*\amalg [0]$, where [0] is the terminal category considered as a multicategory, which serves as the monoidal unit for $\wedge$.  The unit for the monoid E is given by the natural map $u\rightarrow E$.  In particular the multifunctor $E\cong u\wedge E\rightarrow E\wedge E$ is an inverse isomorphism to the monoidal product $E\wedge E\rightarrow E$.  Now we will show that for any based multicategory $N$, the map $-\wedge E:\mult(\N, \M)\rightarrow \mult(\N\wedge E, \M\wedge E)$ is an isomorphism.  This will give a chain of natural isomorphisms \[\mult(\N, \M)\cong \mult(\N\wedge E, \M\wedge E)\cong \mult(\N\wedge E, \M)\cong \mult(\N, \M^E) \]
which shows that $\M\cong \M^E$.

Our proposed inverse to $-\wedge E$ is the map $\phi:\mult(\N\wedge E, \M\wedge E)\rightarrow \mult(\N, \M)$ which sends $\N\wedge E\xrightarrow{f} \M\wedge E$ to the composite \[\N\cong \N\wedge u\rightarrow \N\wedge E\xrightarrow{f}\M\wedge E\cong \M\]

Let $f:\N\rightarrow \M$ be any multifunctor. The composite $\phi\circ (-\wedge E)$ takes $f$ to the multifunctor \[\N\rightarrow \N\wedge E\xrightarrow{f\wedge E} \M\wedge E\xrightarrow{\cong}\M \]

By naturality the square

\begin{center}
\begin{tikzcd}
\N\ar[d, "f"']\ar[r]   &   \N\wedge E\ar[d, "f\wedge E"] \\
\M\ar[r]   &   \M\wedge E
\end{tikzcd}
\end{center}
commutes and so we can replace the composite with \[\N\xrightarrow{f}\M\rightarrow \M\wedge E\cong\M \]
But the composite of the last two multifunctors is the identity because the two multifunctors are inverse isomorphisms, as is spelled out in \cite[Thm 5.1]{EM}, and so the composite reduces to just $f$.

The exact same naturality argument proves that the composite $-\wedge E\circ \phi$ is also the identity.
\end{proof}

\begin{lem}\label{fullfaith}
The fully faithful inclusion $Mod_E\hookrightarrow \mult$ admits left and right adjoints
\begin{center}
    \begin{tikzcd}
    Mod_E\ar[r, hook,""{name=A, below}] & \mult\ar[l,bend left, "(-)^E",""{name=B,above}] \ar[from=A, to=B, symbol=\dashv]\ar[l,bend right,  "E\wedge-"',""{name=C,above}] \ar[from=A, to=B, symbol=\dashv]\ar[ from=C, to=A, symbol=\dashv]
    \end{tikzcd}
\end{center}
\end{lem}
\begin{proof}
For any based multicategory $\cN$ and E-module $\cM$ there are isomorphisms. \[ Mod_E(\cN\wedge E, \cM)\cong \mult(\cN\wedge E, M)\cong \mult(\cN, \cM^E)\cong \mult(\cN, \cM)\] where the first isomorphism follows because $Mod_E\hookrightarrow \mult$ is fully faithful, treating $\cN\wedge E$ as an E-module by multiplying in the E coordinate, and the third is given by Lemma \ref{modstruc1}.  Likewise we have isomorphisms
\[ \mult(\cM, \cN)\cong \mult(\cM\wedge E, \cN)\cong \mult(\cM, \cN^E)\cong Mod_E(\cM, \cN^E) \]
which prove the desired adjunctions.
\end{proof}
From the definition of (co)reflective this proves
\begin{cor}
The category $Mod_E$ is a reflective and coreflective subcategory of the category of based multicategories.
\end{cor}

\begin{rem} 
The right adjoint $(-)^E$ can be though of as taking a multicategory to the category of modules over the basepoint monoid, in the sense that the collection of multifunctors $\mult(E, \M)$ is the same as the collection of objects of $\M$ which have a specified module structure over the basepoint. In this way one can think of an E-module as something like a subcategory of a category of modules over a monoid.
\end{rem}

\begin{ex}
    Let $R$ be a commutative ring. The category $Mod_R$ of $R$-bimodules, considered as a multicategory with the symmetric monoidal structure $\tensor_R$, carries the structure of an E-module.
\end{ex}

To obtain our model structure on $Mod_E$, we will describe a left adjoint to the functor $NJ = \nmult(E^\bullet, -)$, and we will use the resulting adjunction to transfer the stable Q-model structure.  To do this we will need to know that this left adjoint lands naturally in $Mod_E$.

We let the left adjoint of the nerve functor be denoted $ho$.

\begin{lem}
    The functor $NJ\colon Mult*\to \Gamma\textit{-sSet}$ admits a left adjoint which lands in $Mod_E$.
\end{lem}
    \begin{proof}
    This follows by a direct computation. Let $X$ be in $\Gamma\textit{-sSet}$ and $\cM$ a based multicategory.
   \[\begin{array}{rcl} \Gamma\textit{-sSet}(X, NJ(\cM))&\cong& \int_{m\in \Gamma}sSet_*(X_m, \mult(E^m, N\cM)) \\
    &\cong& \int_{m\in \Gamma}\mult(ho(X_m)\wedge E^m, \cM) \\
    &\cong &\mult(\int^{m\in \Gamma}ho(X_m)\wedge E^m, \cM) \\ 
   & = & \mult(X\tensor_{\Gamma} E^*, \cM) 
    
    \end{array}\]
    
    To see that the functor $X\tensor_{\Gamma_*} E^*$ lands in $Mod_E$, we use \cite[Thm 5.12]{EM} which says that for each $m\in \Gamma$, $E^m$ is an E-module, and since the monoidal product $\wedge$ in $\mult$ preserves colimits, we can define the module structure map \[ X\tensor_{\Gamma} E^* \wedge E \to X\tensor_{\Gamma} E^*\] on each component $X_m\wedge E^m$ using the module structure maps of $E^m$.
    \end{proof}

\section{The Model Structure on $Mod_E$}
The aim of this section is to use Kan's Transfer Theorem \ref{theoremtransferredmodelstructure} to prove

\begin{thm}\label{modelstructure}
   The transferred model structure on $Mod_E$ along the right adjoint $NJ$ from the stable Q-model structure exists.
\end{thm}

We will call a based multifunctor a (stable) strict fibration or a (stable) strict weak equivalence if it is taken by $NJ$ to one.

As with most transfers of model structure, the crux of the argument is in verifying the condition (2) of Theorem \ref{theoremtransferredmodelstructure}. To show this we will appeal to Quillen's path object argument, a proof of which is given in Proposition \ref{path} for the readers convenience.

\begin{thm}[Quillen]
   Condition (2) of the Kan Transfer Theorem holds if the proposed model structure admits
   \begin{itemize}
       \item A functorial fibrant replacement functor.
       \item For fibrant $X$, functorial factorization of the diagonal $X\rightarrow X\times X$ into a weak equivalence followed by a fibration.
   \end{itemize}
\end{thm}

The goal then is to prove the existence of functorial path objects for stably fibrant E-modules. We begin by noting

\begin{lem}
    The category $Mod_E$ is complete and cocomplete.
\end{lem}
\begin{proof}
   As a bireflective subcategory, $Mod_E$ is simultaneously a category of algebras for the monad $E\wedge-$ and a category of coalgebras for the comonad $(-)^E$.  Hence the fully faithful inclusion $Mod_E\hookrightarrow \mult$ creates all limits and colimits. The claim now follows because $\mult$ is bicomplete.
\end{proof}

\begin{lem}
    The right adjoint $NJ$ preserves filtered colimits.
\end{lem}
\begin{proof}
   Recall that the nerve functor preserves filtered colimits because it is given levelwise by $Cat( [n], -)$ and $[n]$ is compact.  Likewise, $J$ preserves filtered colimits because the powers $E^n$ are compact.
\end{proof}

Our path object factorizations will be based on a multicategory which parameterizes module morphisms. Let $I$ denote the based multicategory consisting of a basepoint monoid, two modules, and an arrow between them which preserves the module structure. Rigorously I consists of objects 0, 1, 2 such that the full sub-multicategories on $\{0,1\}$ and $\{0,2\}$ are isomorphic to E, and such that the multimorphism sets $I(a_1,\dots,a_n:2)$ are empty unless exactly one of the $a_i$ is 1 or 2 and the rest are 0 in which case they are singleton sets. In particular there is a unique 1-arrow $1\rightarrow2$.

\begin{lem}\label{arrow}
    For any based multicategory $\cM$, there is an isomorphism of categories $\mult(I, \cM )\cong \mult(E,\cM)^{[1]}$.
\end{lem}
\begin{proof}
An object of $\mult(I, \cM )$ is a multifunctor that parameterizes a module map between modules. An object of $\mult(E,\cM)^{[1]}$ is a multi-natural transformation between two multifunctors $E\rightarrow\cM$ which each pick out a module object in $\cM$. The data of such a transformation is a 1-arrow between the two modules in $\cM$.  The multinaturality condition says precisely that the 1-arrow preserves the module structure. Thus the two categories have the same objects.  A similar check shows that the categories have the same arrows in a compatible fashion, giving the claimed isomorphism.
\end{proof}

Before we prove the existence of path factorizations we recall Quillen's Theorem A.

\begin{thm}[Quillen]
If $F:\cC\rightarrow \cD$ is a functor such that for each object $d\in\cD$ the simplicial set $N(d/F)$ is contractible, then F induces a homotopy equivalence $NF:N\cC\rightarrow N\cD$.
\end{thm}

\begin{prop}
For a stably fibrant E-module $\cM$ the sequence $E\vee E \hookrightarrow I\rightarrow E$ induces a factorization
\[ \cM\xrightarrow{\simeq} \cM^I \twoheadrightarrow \cM\times \cM \]
into a strict weak equivalence followed by a stable fibration.
\end{prop}
\begin{proof}
To show that the first map is a strict weak equivalence we must show that for each $n$ \[ \nmult(E^n, \cM)\rightarrow \nmult(E^n, \cM^I)\] is a weak equivalence. By adjunction this map has the form \[ \nmult(E, \cM^{E^n})\rightarrow \nmult(I, \cM^{E^n})\] and so it suffices to show that for any E module $\cM$  \[ \nmult(E, \cM)\xrightarrow{0^*} \nmult(I, \cM)\] is a strict weak equivalence. Let $A:E\rightarrow \cM$ be any object of the left hand side, representing a module object $A\in \cM$, then the multifunctor $1_A:I\rightarrow \cM$ choosing the identity at $A$ is an initial object in the slice category $A/0^*$.  This implies that the slice is contractible, and so by Quillen's Theorem A the map $0^*$ is a weak equivalence.

For the second map, recall that because the stable Q-model structure is defined by a localization, every strict fibration between stably fibrant objects is a stable fibration, so it suffices to show that the map is a strict fibration and that the domain and codomain are stably fibrant. Recall that this means we need to show that they are levelwise Kan complexes, that they satisfy the Segal condition, and that their $\pi_0$ is a group. By Lemma \ref{arrow}, \[\nmult(E^*, \cM^I)\cong \nmult(E^*, \cM)^{\Delta[1]}\]  This object is clearly levelwise Kan, and  satisfies \[\pi_0\nmult(E, \cM)^{[1]}\cong \pi_0\nmult(E, \cM)\]  The Segal condition comes from cotensoring the assumed weak equivalences \[ \nmult(E^n, \cM)\xrightarrow{\simeq} \nmult(E, \cM)\times\cdots\times \nmult(E,\cM) \] by $\Delta[1]$, which remains a weak equivalence because the left and right hand side are Kan complexes, and so this object is stably fibrant. Finally the object $\nmult(E^*, \cM)\times \nmult(E^*, \cM)$ is stably fibrant because this property is closed under products.
\end{proof}

Now we can conclude the existence of the model structure by using Quillen's path object argument, which we provide a proof of for convenience.

\begin{prop}\label{path}
Every multifunctor with the left lifting property against all stable fibrations is a weak equivalence.
\end{prop}
\begin{proof}
   Let $f:X\rightarrow Y$ lift against stable fibrations and choose stably fibrant replacements  $\tilde{X}, \tilde{Y}$ in $Mod_E$. Consider the following diagram
   \begin{center}
   \begin{tikzcd}
    X\ar[dr, dotted, "i"]\ar[r, "f"]\ar[dd, "(1\times f)"'] & Y\ar[r, "\simeq"] & \tilde{Y}\ar[d, "\simeq"]\\
    & Pf\ar[dr, phantom, "\lrcorner", very near start]\ar[r]\ar[d, two heads] & \tilde{Y}^I\ar[d, two heads] \\
    X\times Y\ar[r, "\simeq"] & \tilde{X}\times \tilde{Y}\ar[r, "(f\times 1)"]\ar[d, two heads] & \tilde{Y}\times \tilde{Y}\ar[d, two heads] \\
    & \tilde{X}\ar[r, "\tilde{f}"'] & \tilde{Y}
   \end{tikzcd}
   \end{center}
   
   The map $\pi: Pf\rightarrow \tilde{X}\times\tilde{Y}\rightarrow \tilde{X}$ is a weak equivalence because it is the pullback of the trivial fibration $\tilde{Y}^I\twoheadrightarrow \tilde{Y}$. Because $\pi i: X\rightarrow \tilde{X}$ is a weak equivalence it follows that $i$ is as well.  Because the map $Pf\rightarrow \tilde{X}\times\tilde{Y}\rightarrow \tilde{Y}$ is a fibration there is a lift in the diagram
   
   \begin{center}
       \begin{tikzcd}
       X\ar[r, "\simeq"]\ar[d, "f"']  &   Pf\ar[d, two heads] \\
       Y\ar[r, "\simeq"']\ar[ur, dotted]    &   \tilde{Y}
       \end{tikzcd}
   \end{center}
   
   and so by the two of six property the map $f$ is a weak equivalence.
\end{proof}

The choice to restrict from $\mult$ to $Mod_E$ was made in order to use the path object argument, which requires a factorization of the diagonal $\M\rightarrow \M\times \M$. The issue with based multicategories which are not E-modules is that cotensoring with $E\vee E\hookrightarrow I\rightarrow E$ provides a factorization instead of the diagonal $\M^E\rightarrow \M^E\times \M^E$, which is different as Proposition \ref{modstruc1} does not apply in general.

Note however that the functor $J$ induces a levelwise isomorphism
\[J(\M^E)\xrightarrow{\cong} J(\M)\] because every multifunctor out of an E-module factors through $\M^E$, and so no $K$-theoretic information is being lost, in a very strict sense, upon this restriction.

If we insist on working with the category $\mult$, we can at the least lift the model structure on $Mod_E$ along the adjunction

\begin{center}
    \begin{tikzcd}
    \mult\ar[r,"(-)^E"',""{name=A, below}] & Mod_E\ar[l,hook, bend right, ""{name=B,above}] \ar[from=B, to=A, symbol=\dashv]
    \end{tikzcd}
\end{center}
to a \textit{semi}-model structure on $\mult$. 

\begin{defn}\label{semimodeldef}
A semi-model structure on a category $\cM$ is a choice of three classes of arrows: weak equivalences, fibrations, and cofibrations, which satisfy the usual model category axioms 
\begin{enumerate}
    \item[(M1)] $\cM$ is bicomplete.
    \item[(M2)] The weak equivalences are closed under the two of three property.
    \item[(M3)] The three classes of maps are closed under retracts.
\end{enumerate}
Except we replace axioms M4, M5 with the following weaker axioms
\begin{enumerate}
    \item[(M4')] 
        \begin{itemize}
            \item The fibrations have the right lifting property with respect to trivial cofibrations with cofibrant domain.
            \item The trivial fibrations have the right lifting property with respect to cofibrations with cofibrant domain.
        \end{itemize}
    \item[(M5')]
        \begin{itemize}
            \item Every morphism with cofibrant domain factors into a cofibration followed by a trivial fibration.
            \item Every morphism with cofibrant domain factors into a trivial cofibration followed by a fibration.
        \end{itemize}
    \item[(M6')] The initial object is cofibrant.
\end{enumerate}
\end{defn}

The idea is that once we restrict to just the cofibrant objects of a semi-model structure we are effectively working in a model category, with all of our familiar model categorical tools available to us.  For this reason a semi-model structure is often sufficient to do homotopy theory. For more information on the power of semi-model structures, see for example \cite{F}.

We will need the following variant of the Kan transfer theorem for semi-model structures from \cite[Thm 12.1.4]{F}.

\begin{thm}[Fresse]\label{fresse}
   Let $\cM$ be a cofibrantly generated model category with a  set of generating cofibrations $I$ and a set of generating trivial
cofibrations $J$, let $\cN$ be a complete and cocomplete category, and let
$$F\colon\cM\rightleftarrows\cN\colon U$$
be an adjunction. Furthermore let $F(\cN_c) = \{F(i): i\text{ is a cof in } \cN\}$. 

Suppose that the following conditions hold.
\begin{enumerate}
\item The right adjoint preserves filtered colimits.

\item Letting $\emptyset\in \cM$ denote the initial object, $UF(\emptyset)$ is cofibrant

\item For any $F(\cN_c)$-cell complex $A$ with $U(A)$ cofibrant, and any pushout

\begin{center}
    \begin{tikzcd}
    F(K)\ar[r]\ar[d, "F(i)"']    &   A\ar[d, dotted, "f"]   \\
    F(L)\ar[r, dotted]    &   B
    \end{tikzcd}
\end{center}
$U(f)$ is a (trivial) cofibration when $i$ is a (trivial) cofibration with cofibrant domain.

\end{enumerate}
Then, there is a cofibrantly generated semi-model category structure on $\cN$ in which
\begin{itemize}
\item the set $F(I)$ is
a set of generating cofibrations,
\item the set $F(J)$ is a set of generating acyclic cofibrations,
\item $U$ creates weak equivalences and fibrations and preserves cofibrations.
\end{itemize}
\end{thm}

\begin{thm}
   The adjunction
   \begin{center}
    \begin{tikzcd}
    \mult\ar[r,"(-)^E"',""{name=A, below}] & Mod_E\ar[l,hook, bend right, ""{name=B,above}] \ar[from=B, to=A, symbol=\dashv]
    \end{tikzcd}
    \end{center}
endows $\mult$ with the structure of a cofibrantly generated semi-model structure.
\end{thm}
\begin{proof}
The right adjoint $(-)^E$ preserves filtered colimits, and the initial object is clearly cofibrant as it is (trivially) in $Mod_E$.  Therefore we need only check condition (3) of \ref{fresse}, which in our case reads as follows: given any (trivial) cofibration of E-modules with cofibrant domain, $i:K\rightarrow L$, and any cell complex $A$, the morphism $f$ in the pushout
\begin{center}
    \begin{tikzcd}
    K\ar[r]\ar[d, "i"']    &   \cA\ar[d, dotted, "f"]   \\
    L\ar[r, dotted]    &   \cB
    \end{tikzcd}
\end{center}
is a (trivial) cofibration. Now by \cite[12.1.3]{F} $\cA$ is necessarily cofibrant, however note that for any based multicategory  $\cM$, the inclusion $\cM^E\hookrightarrow \cM$ is taken by $(-)^E$ to an isomorphism and so is a trivial fibration. So since $\emptyset\rightarrow \cA$ is a cofibration with cofibrant domain, there is a lift 

\begin{center}
    \begin{tikzcd}
        &   \cA^E\ar[d]   \\
    \cA\ar[r, equals]\ar[ur, dotted]    &  \cA
    \end{tikzcd}
\end{center}
which implies an isomorphism $\cA^E\cong \cA$.  Hence $\cA$ is necessarily an E-module. Thus the above pushout takes place entirely in $Mod_E$ (which in particular is closed under pushouts), and so $f$ is a (trivial) cofibration because $Mod_E$ is a model category.
\end{proof}

\begin{rem}
It may be that the above semi-model structure is a fully fledged model structure, but the proof of this would require more careful analysis since we cannot make use of our particular path object factorization for $Mod_E$. 
\end{rem}

Now we focus on recapturing Thomason's theorem in our setting. Thomason deals specifically with the equivalence of homotopy categories given by Segal's functor $NK\colon SymMonCat\rightarrow \Gamma\textit{-Cat}\rightarrow \Gamma\textit{-sSet}$.  Since we have changed our consideration to $\mult$ and again to $Mod_E$ we will need the following proposition relating these three categories. In each case, by the stable homotopy category we refer to the homotopy category of the relative category obtained by equipping each category with stable weak equivalences created by the suitable $K$-theory functor.
\begin{prop}\label{euivstable}
The stable homotopy categories of the following three categories are equivalent
\begin{itemize}
    \item The category $SymMonCat$ of symmetric monoidal categories and strict monoidal functors.
    \item The category $\mult$ of based multicategories and based multifunctors.
    \item The ccategory $Mod_E$ of E-modules and based multifunctors.
\end{itemize}
\end{prop}
\begin{proof}
The equivalence of the first two follows from the discussion of Section 3 in \cite{M} using the comonadic adjunction
\begin{center}
    \begin{tikzcd}
    \mult\ar[r, bend right,""{name=A, below}] & SymMonCat\ar[l, bend right, ""{name=B,above}] \ar[from=A, to=B, symbol=\dashv]
    \end{tikzcd}
\end{center}
described in Section 4 of \cite{EM}. The equivalence of the homotopy categories of $\mult$ and $Mod_E$ follows from the fact that $J(\cM^E)\rightarrow J(\cM)$ is a levelwise isomorphism, so the adjunction 
\begin{center}
    \begin{tikzcd}
    \mult\ar[r,bend right,"(-)^E"',""{name=A, below}] & Mod_E\ar[l,hook, bend right, ""{name=B,above}] \ar[from=B, to=A, symbol=\dashv]
    \end{tikzcd}
    \end{center}
gives the desired equivalence at the level of homotopy categories.
\end{proof}

We now obtain as our second main result the desired upgrade of Thomason's theorem.
\begin{thm}\label{main2}
   The right Quillen functor $NJ: Mod_E\rightarrow \Gamma\textit{-sSet}$ is a Quillen equivalence.
\end{thm}
    \begin{proof}
    By the transfer theorem, the functor is already right Quillen. Because $NJ$ creates weak equivalences, the composite \[Mod_E\xrightarrow{NJ} \Gamma\textit{-sSet}\rightarrow Ho(\Gamma\textit{-sSet})\] factors through the derived functor \[Ho(NJ)\colon Ho(Mod_E)\rightarrow Ho(\Gamma\textit{-sSet})\] By \cite{EM} there is a natural isomorphism 
    \begin{center}
        \begin{tikzcd}
            \mult\ar[to=A, phantom, near start, "\cong"]\ar[r, "NJ"]   &   \Gamma\textit{-sSet}    \\
            SymMonCat\ar[u, "i"]\ar[ur, "NK"', ""{name=A}]  &
        \end{tikzcd}
    \end{center}
    Likewise there is a natural isomorphism by Lemma \ref{fullfaith}
    \begin{center}
        \begin{tikzcd}
            Mod_E\ar[to=A, phantom, "\cong"]\ar[r, "NJ"]   &   \Gamma\textit{-sSet}    \\
            \mult\ar[u, "(-)^E"]\ar[ur, "NJ"', ""{name=A}]  &
        \end{tikzcd}
    \end{center}
    
    Putting these together we have a natural isomorphism \[NK\cong NJ\circ (i(-))^E \]  By \cite{Thom} or \cite{M} the functor $Ho(NK)$ is an equivalence on homotopy categories, and by Proposition \ref{euivstable} $Ho((i(-))^E)$ is as well. Because all of the functors under consideration create weak equivalences the derived functors compose up to isomorphism, and thus it follows from the two out of three property that $Ho(NJ)$ is an equivalence. 
    \end{proof}

At the level of $\mult$ we immediately have the following corollary, treating a model structure as a special type of semi-model structure.
    \begin{cor}
        The functor $NJ:\mult\rightarrow \Gamma\textit{-sSet}$ is a right Quillen equivalence of semi-model structures.
    \end{cor}
    
\section{Symmetric Monoidal Groupoids}
In this final section we characterize the fibrant objects in $Mod_E$ in order to give an application of Theorem \ref{modelstructure}.

To better understand the categories $\mult(E^n, \M)$ we need to understand the multicategories $E^n$.  Fortunately these turn out to have a very natural description. Consider for example $E^2$.  We denote the objects $(00, 10, 01, 11)$.  The multi-homsets in the product are the product of multi-homsets, and so we see in particular
\begin{itemize}
    \item The object $00$ is a monoid.
    \item Every object is a module over $00$.
    \item There are no interesting 1-arrows.
    \item There is an interesting 2-arrow $(01, 10)\rightarrow 11$.
\end{itemize}

The data of a based multifunctor $E^2\rightarrow \M$ is given then by three modules $A_{01},A_{10}, A_{11}$ and a 2-arrow $(A_{01}, A_{10})\rightarrow A_{11}$. If we consider $E^3$ we have a very similar structure: There is a monoid object $000$ together with 8 modules $ijk$ where each of $i,j,$ and $k$ takes the value 0 or 1.  There are still no 1-arrows, and there is a distinguished 3-arrow $(100, 010, 001)\rightarrow (111)$.  However there are now interesting 2-arrows, for example $(011, 100)\rightarrow 111$ and $(010, 100)\rightarrow 110$.  In fact we see that there will be an m-arrow \[(i_1j_1k_1,\dots, i_mj_mk_m)\rightarrow ijk \] precisely when $\Sigma i_l = i, \Sigma j_l = j,$ and $\Sigma k_l = k$.  The situation is similar for general $n$.  In particular letting $1_i\in E^n$ denote the object which is 0 in every coordinate except the $i^{th}$,  $E^n$ contains a distinguished $n$-arrow \[(1_1,\dots, 1_n)\rightarrow 11\dots1\]
A multifunctor $E^n\rightarrow \M$ picks out exactly this data in $\M$. An enlightening example is given by the following.
\begin{ex}
    Let $\M$ be a symmetric monoidal category.  The following diagram represents a particularly special multifunctor $E^3\rightarrow \M$, not counting the basepoint of $\M$.
    
    \begin{center}
        \begin{tikzcd}
        & A\tensor B\tensor C   & \\
        A\tensor B  &   A\tensor C & B\tensor C \\
        A   &   B   &   C
        \end{tikzcd}
    \end{center}
\end{ex}
The distinguished 3-arrow is given by an isomorphism \[A\tensor B\tensor C\xrightarrow{\cong} A\tensor B\tensor C\]The 2-arrows $(011, 100)\xrightarrow{\cong} 111$ and $(010, 100)\rightarrow 110$ correspond to the isomorphisms  \[(B\tensor C)\tensor A\xrightarrow{\cong} A\tensor B\tensor C \hspace{1cm} B\tensor A\xrightarrow{\cong} A\tensor B \]
In fact, for any symmetric monoidal category $\M$, the $\Gamma$-space $NJ(\M)$ satisfies the Segal conditions since the map \[\nmult(E^n, \M)\xrightarrow{p_n} \nmult(E, \M)\times\cdots\times \nmult(E, \M)\] is a weak equivalence by Quillen's Theorem A: for each object of the right hand side, represented as a collection of modules $(A_1,\dots, A_n)$, the slice $(A_1,\dots, A_n)/p_n$ has an initial object given as above.

\vspace{0.25cm}

Every symmetric monoidal groupoid whose set of objects forms an abelian group is fibrant in our model structure. Conversely we have 
\begin{prop}
Every fibrant object in $Mod_E$ is stably equivalent to a symmetric monoidal groupoid.
\end{prop}
\begin{proof}
Let $\cM$ be fibrant. The Segal conditions, combined with levelwise fibrancy of $NJ(\cM)$ imply that there are maps \[\nmult(E, \M)^{\times n}\xrightarrow{\simeq} \nmult(E^n, \M)\rightarrow\nmult(E,\M) \]
and hence also functors on the level of the underlying category of $\M$
\[\tensor^n:\M\times\cdots\times \M\rightarrow \M \]
which are compatible and endow the underlying category of $\M$ with a symmetric monoidal structure.  We refer to the resulting symmetric monoidal category as $\mtens$, which we consider as a multicategory based at the unit with \[\mtens(a_1,\dots,a_n: b) = \M(a_1\tensor\cdots\tensor a_n, b)\]  Now for any  $a_1,\dots,a_n\in \M$ the nerve of the slice $N(a_1\tensor\cdots\tensor a_n/\M)$ is contractible because it has an initial object.
The Segal conditions give that each of the fibers

\begin{center}
\begin{tikzcd}
\nmult(E^n, \M)_{(a_1,\dots,a_n)}\ar[r]\ar[d, "\simeq"']\ar[dr, phantom, "\lrcorner", very near start, bend right = 10em]   &   \nmult(E^n, \M)\ar[d, "\simeq"]\\
\Delta[0]\ar[r]   &   \nmult(E, M)\times\cdots\times \nmult(E,\M)
\end{tikzcd}
\end{center}
is contractible as well, and so there are equivalences for each $a_1,\dots,a_n$ \[\nmult(E^n, \M)_{(a_1,\dots,a_n)} \simeq N(a_1\tensor\cdots\tensor a_n/\M) \]
However this must be true as well for $\mtens$ since it satisfies the Segal condition, and since $N(a_1\tensor\cdots\tensor a_n/M)\cong N(a_1\tensor\cdots\tensor a_n/\mtens)$ we have equivalences \[\nmult(E^n, \M)_{(a_1,\dots,a_n)} \simeq \nmult(E^n, \mtens)_{(a_1,\dots,a_n)}   \]
which assemble into equivalences \[\nmult(E^n, \M) \simeq \nmult(E^n, \mtens)\]

Now there is an inclusion of multicategories $\mtens\rightarrow \M$ which is the identity on objects and acts on $n$-arrows $\M(a_1\tensor\cdots\tensor a_n, b)\rightarrow M(a_1,\dots, a_n:b)$ by pulling back along the canonical $n$-arrow $a_1,\dots, a_n\rightarrow a_1\tensor\cdots\tensor a_n$.  By the above argument this multifunctor is taken by $NJ$ to a levelwise equivalence.
\end{proof}
       
In light of the Quillen equivalence of Theorem \ref{main2} every connective spectrum is stably equivalent to the $K$-theory of some cofibrant-fibrant object of $Mod_E$, this proves immediately a refinement of Thomason's theorem.

\begin{thm}\label{symgroup}
   Every connective spectrum arises up to stable equivalence as the $K$-theory of a symmetric monoidal groupoid.
\end{thm}

As a final remark, note that any fibrant object satisfies also that $\pi_0\cM$ is a group. This would seem to imply that $\mtens$ is a Picard groupoid, a symmetric monoidal groupoid in which every object is invertible with respect to the monoidal product, except that it is not \textit{a priori} clear that this multiplicative structure on $\pi_0$ is related to the monoidal structure on $\mtens$.
             


\begin{thebibliography}{9}

\bibitem[BF78]{BF} A. K. Bousfield and E. M. Friedlander, Homotopy theory of $\Gamma$-spaces, Spectra and Bisimplicial Sets, Springer Lecture Notes in Math., Vol. 658, Springer, Berlin, 1978.


\bibitem[EM06]{EM2} A. D. Elmendorf and M. A. Mandell, Rings, Modules, and Algebras in Infinite Loop Space Theory, Advances in Mathematics, 2006.

\bibitem[EM09]{EM} A. D. Elmendorf and M. A. Mandell, Permutative Categories, Multicategories, and Algebraic K-theory, Algebraic and Geometric Topology, 2009.


\bibitem[F09]{F} B. Fresse, Modules over Operads and Functors, Lecture Notes in Mathematics, Springer, 2009.

\bibitem[M10]{M} M. A. Mandell, An Inverse K-Theory Functor, Documenta Mathematica, 2010.

\bibitem[S72]{S}G. Segal, Categories and Cohomology Theories, Topology, 1972.

\bibitem[S99]{Sc} S. Schwede, Stable Homotopical Algebras and $\Gamma$-spaces, Mathematical Proceedings of the Cambridge Philisophical Socieity, 1999.

\bibitem[T80]{Thom2} R. W. Thomason, Cat as a Closed Model Category, Cahiers Topologie et Geometrie Differentielle, 1980.

\bibitem[T95]{Thom} R. W. Thomason, Symmetric Monoidal Categories Model all Connective Spectra, Theory and Applications of Categories, 1995.

\end{thebibliography}
\end{document}